\documentclass[10pt,a4paper]{amsart}
\usepackage[cp1250]{inputenc}
\usepackage[OT4]{fontenc}
\usepackage{fullpage}
\usepackage{amsfonts}
\usepackage{hyperref}
\usepackage{amssymb}
\usepackage{graphicx}
\usepackage{color}
\usepackage{enumerate}

\newtheorem*{lem}{Lemma}

\newtheorem*{twm}{Main Theorem}
\newtheorem*{prop}{Proposition}

\theoremstyle{definition}
\newtheorem*{df}{Definition}

\def\C{\mathbb{C}}

\def\t{\widetilde}

\def\O{\mathcal{O}}

\def\Omega{\varOmega}
\def\Phi{\varPhi}
\def\Psi{\varPsi}
\def\epsilon{\varepsilon}

\begin{document}

\title{A note on singularity of fibers of singular sets}
\author[M. Zaj\k{e}cka]{Ma{\l}gorzata Zaj\k{e}cka}
\address{\textnormal{Jagiellonian University\newline
\indent Faculty of Mathematics and Computer Science\newline
\indent Institute of Mathematics\newline
\indent {\L}ojasiewicza 6\newline
\indent 30-348 Krak\'ow}}

\address{\textnormal{Cracow University of Technology\newline
\indent Faculty of Physics, Mathematics and Computer Science\newline
\indent Institute of Mathematics\newline
\indent Warszawska 24\newline
\indent 31-155 Krak\'ow\newline
$ $}}

\email{malgorzata.zajecka@gmail.com}
\subjclass[2010]{}
\begin{abstract}
We present a general theorem on fibers of singular sets.
\end{abstract}
\maketitle

\section{Introduction and prerequisities}

In this paper we show an elementary proof of the following result.

\begin{twm}
Let $D_1$ be a connected Josefson manifold (i.e.~$D_1$ is a complex manifold
such that every locally pluripolar set in $D_1$ is globally pluripolar) of
dimension $n_1$ and let $D_2$ be a complex manifold of dimension $n_2$.
Let $\Omega\subset D_1\times D_2$ be a domain and let $M\subset\Omega$ be
a singular set with respect to the family
$\mathcal{F}\subset\O(\Omega\setminus M)$ such that the set $\{a_1\in D_1:\text{
the fiber }M_{(a_1,\cdot)}\text{ is not pluripolar}\}$ is pluripolar in $D_1$.
Then there exists a pluripolar set $Q\subset D_1$ such that for every $a_1\in
\pi_{D_1}(\Omega)\setminus Q$ the fiber $M_{(a_1,\cdot)}$ is singular in
$\Omega_{(a_1,\cdot)}$ with respect to the family
$\mathcal{F}_a:=\{f(a_1,\cdot):f\in\mathcal{F}\}\subset\mathcal{O}(\Omega_{(a_1,\cdot)}\setminus
M_{(a_1,\cdot)})$, where $\pi_{D_1}(\Omega)$ denotes the projection of $\Omega$
to $D_1$ and for $B\subset D_1\times D_2$ and $a_1\in D_1$ we put
$B_{(a_1,\cdot)}:=\{a_2\in D_2: (a_1,a_2)\in B\}$.
\end{twm}

Similar property of singular sets in Riemann domains was originally
proven in \cite{JarPfl 2006} (Theorem 2.2). A stronger version for
Riemann regions of holomorphy was published in \cite{JarPfl 2011} (Proposition
9.1.4).

\begin{prop}
Let $D_1$, $D_2$ be Riemann domains over $\C^{n_1}$ and $\C^{n_2}$,
respectively. Let $\Omega\subset D_1\times D_2$ be a Riemann region of holomorphy and let
$M\subset\Omega$ be a singular set with
respect to the family $\mathcal{F}\subset\O(\Omega\setminus M)$, such that the set $\{a_1\in D_1:\text{
the fiber }M_{(a_1,\cdot)}\text{ is not pluripolar}\}$ is pluripolar in $D_1$. Then there
exists a pluripolar set $Q\subset D_1$ such that for any $a_1\in
\pi_{D_1}(\Omega)\setminus Q$ the fiber $M_{(a_1,\cdot)}$ is singular with
respect to the family $\mathcal{F}_{a_1}=\{f(a_1,\cdot):f\in\mathcal{F}\}$.
\end{prop}

Above proposition with slightly stronger assumption that $M$ is a relatively
closed pluripolar set in $D_1$ is exactly Proposition 9.1.4 from \cite{JarPfl
2011}. To see that it remains true in our case, the Reader is asked to follow the original proof.

Because Riemann domains are special case of Josefson manifolds, the
natural question was whether the property holds for singular sets in arbitrary
complex manifolds.

The original version of this result was one of the key properties used in the
proof of strong theorems concerning extensions of functions separately
holomorphic on different kinds of generalized crosses with singularities on
Riemann domains - see \cite{JarPfl 2007}, \cite{JarPfl 2011}, \cite{Zaj}.

Now, recall
the definition of singular set.
For an $n$-dimensional complex manifold $X$ let $M$ be a closed subset of $X$ such that for any domain
$\Omega\subset X$ the set $\Omega\setminus M$ is connected and dense in
$\Omega$ (for instance, let $M$ be a pluripolar set). Let $\mathcal{F}$ be a
family of functions holomorphic on $X\setminus M$.

\begin{df}\label{def_zb_sing}
A point $a\in M$ is called \emph{singular with respect to the family
$\mathcal{F}$}, if for any open connected neighborhood
$U_a$ of the point $a$ there exists a function $f\in\mathcal{F}$, that does not
extend holomorphically to $U_a$.

We call $M$ \emph{singular with respect to the family $\mathcal{F}$}, if every
point $a\in M$ is singular with respect to $\mathcal{F}$.
\end{df}

We have the following property.

\begin{lem}
Let $X_1$, $X_2$ be complex manifolds and let $\Phi:X_1\to X_2$ be a
biholomorphism. Let $M\subset X_1$ be a singular set with respect to the family $\mathcal{F}\subset\O(X_1\setminus M)$. Then the
set $N=\Phi(M)\subset X_2$ is singular with respect to the family
$\t{\mathcal{F}}:=\{f\circ\Phi^{-1}:f\in\mathcal{F}\}\subset\O(X_2\setminus N)$.
\end{lem}

\begin{proof}
From the definition, $N$ is a closed subset of $X_2$ such that for any domain
$\Omega\subset X_2$ the set $\Omega\setminus N$ is connected and dense in
$\Omega$. We show that for any $a\in N$ and for any open connected neighborhood
$U_a$ of the point $a$ there exists a function $\t{f}\in\t{\mathcal{F}}$,
that does not extend holomorphically to $U_a$.

Fix an $a\in N$ and define $b:=\Phi^{-1}(a)\in M$. Assume that there exists an
open neighborhood $U_a$ of $a$ such that every function
$\t{f}\in\t{\mathcal{F}}$ extends holomorphically on $U_a$.
Let $V_b:=\Phi^{-1}(U_a)\subset X_1$. Fix $f\in \mathcal{F}$
and define $\t{f}:=f\circ \Phi^{-1}\in\t{\mathcal{F}}$. Then, from our
assumption, $\t{f}$ extends to a function $\t{F}$ holomorphic on $U_a$. Define
$F:=\t{F}\circ\Phi|_{V_b}\in\O(V_b)$. Since $F=\t{F}\circ\Phi=\t{f}\circ\Phi=f$
on the nonempty open set $V_b\setminus M$, we conclude that $F$ is a holomorphic
extension of $f$ to $V_b$ - a contradiction.
\end{proof}

\section{Proof of Main Theorem}

% Fix $a=(a_1,a_2)\in M$, where $a_1\in D_1$, $a_2\in D_2$. Then there exists an
% open connected neighbourhood $U_a=U_{1}\times U_{2}\subset
% \Omega$ of the point $a$, where $U_1\subset D_1$ and $U_2\subset D_2$
% are open connected neighbourhoods of $a_1$ and $a_2$, respectively.

% Let $\Phi_1:U_1\to \C^l$ and $\Phi_2:U_2\to \C^k$ be
% complex maps. Define $z_1:=\Phi_1(a_1)$, $z_2:=\Phi_2(a_2)$, $z=(z_1,z_2)$,
% and $\t{U}_1:=\Phi_1(U_1)$, $\t{U}_{2}:=\Phi_2(U_2)$,
% $\t{U}_z:=\t{U}_1\times \t{U}_2$. Then the
% mapping $\Phi:=(\Phi_1, \Phi_2):U_1\times U_2\to
% \C^l\times\C^k$ is a complex map on $U_a$, $\Phi(a)=z$, and $\Phi(U_a)=\t{U}_z$.
% Moreover, without lost of generality we may assume that both $\t{U}_1$ and $\t{U}_2$ are
% euclidean balls in, respectively, $\C^l$ and $\C^k$. Thus, in particular,
% $\t{U}_z$ is a domain of holomorphy in $\C^l\times \C^k$.

Fix $a=(a_1,a_2)\in M$, where $a_1\in D_1$, $a_2\in D_2$. Let
$\Phi_j:U_j\to\t{U}_j$ be a biholomorphic mapping such that $U_j$ is
and open neighbourhood of $a_j$, $\t{U}_j$ is an Euclidean ball in $\C^{n_j}$,
$\Phi_j(a_j)=0$, $j=1,2$, and $U_a:=U_1\times U_2\subset\Omega$.

Define $\Phi:=(\Phi_1,\Phi_2)$ and $N:=\Phi(M\cap U_a)$,
$\mathcal{F}_a:=\{f|_{U_a}:
f\in\mathcal{F}\}$, $\t{\mathcal{F}}_a:=\{f\circ \Phi^{-1}:
f\in\mathcal{F}_a\}$. Then $\t{\mathcal{F}}_a\subset
\O(\t{U}\setminus N)$ and, from Lemma, $N$ is singular with respect to the family $\t{\mathcal{F}}_a$.
Moreover, the set $\{z_1\in \t{U}_1:\text{ the fiber }N_{(z_1,\cdot)}\text{ is
not pluripolar}\}$ is pluripolar in $\t{U}_1$.

Now, from Proposition there
exists a pluripolar set $\t{Q}_a\subset\C^{n_1}$ such that for any
$w_1\in\pi_{\C^{n_1}}(\t{U}_z)\setminus\t{Q}_a$ the fiber
$N_{(w_1,\cdot)}=\{w_2\in\C^{n_2}:(w_1,w_2)\in N\}$ is singular with respect to
the family $\t{\mathcal{F}}_{w_1}:=\{g(w_1,\cdot):g\in\t{\mathcal{F}}_a\}$.
Define $b_1:=(\Phi_1)^{-1}(w_1)\in\pi_{D_1}(U_a)$,
$Q_a:=(\Phi_1)^{-1}(\t{Q}_a)$.
Then $Q_a$ is pluripolar in $D$ and
$$
N_{(w_1,\cdot)}=\{w_2\in\C^k:\exists b_2\in
U_2:\Phi_2(b_2)=w_2, (b_1,b_2)\in M\}=\Phi_2(M_{(b_1,\cdot)}).
$$
Thus, from Lemma, for any
$b_1\in\pi_{D_1}(U_a)\setminus Q_a$ the fiber $M_{(b_1,\cdot)}$ is singular with
respect to the family $\mathcal{F}_{b_1}:=\{f(b_1,\cdot):f\in\mathcal{F}_a\}$.

% Fix $b'\in\pi_D(U_a)\setminus Q_a$ and $b''\in M_{(b',\cdot)}$. Assume that
% there exists an open neighbourhood $U''_{b''}$ of
% $b''$ such that every function $f(b',\cdot)$, where $f\in\mathcal{F}_a$,
% extends holomorphically to $U''_{b''}$. Fix $f\in\mathcal{F}_a$
% and define $w':=\Phi_{a'}(b')$, $w'':=\Psi_{a'}(b'')$, and $g:=f\circ(\Phi_{a'},
% \Psi_{a''})^{-1}$.
% Then $w''\in N_{(w',\cdot)}$ and $g(w',\circ)=f(b',(\Psi_{a''})^{-1}(\cdot))$
% extend to $\Psi_{a''}(U''_{b''})$ - a contradiction.

From $\{U_a\}_{a\in M}$ we select a countable covering
$\{U_{a_j}\}_{j=1}^\infty$ of the set $M$. Define $Q:=\bigcup_{j=1}^\infty Q_{a_j}\cup \{b_1\in D_1:
M_{(b_1,\cdot)}\text{ is not pluripolar}\}$. Because $D$ is a Josefson manifold, $Q$ is pluripolar in $D$. We show that for any $b_1\in\pi_{D_1}(\Omega)\setminus Q$ the fiber $M_{(b_1,\cdot)}$ is singular with respect to the family
$\mathcal{F}_{b_1}:=\{f(b_1,\cdot):f\in\mathcal{F}\}$.

Fix $b_1\in\pi_{D_1}(\Omega)\setminus Q$, $b_2\in M_{(b_1,\cdot)}$. Assume that
there exists an open neighbourhood $V_{b_2}$ of $b_2$
such that any function $f(b_1,\cdot)$, $f\in\mathcal{F}$, extends
holomorphically on $V_{b_2}$.
Because $(b_1,b_2)\in M$, then there exist $a_j=(a_{j1},a_{j2})$,
$j\in\{1,2,\ldots\}$, such that $(b_1,b_2)\in U_{a_j}=U_{j1}\times U_{j2}$,
where $U_{jk}$ is and open neighbourhood of $a_{jk}$, $k=1,2$. Thus $b_1\in
U_{j1}\setminus Q_{a_j}$ and the functions $f|_{U_{a_j}}(b_1,\cdot)$,
$f\in\mathcal{F}$, extend holomorphically on $V_{b_2}\cap U_{j2}$ - a
contradiction.

\end{document}